\documentclass{amsart}
\usepackage{amsfonts,amssymb,amsmath,amsthm}
\usepackage{url}
\usepackage{enumerate}

\newtheorem{theorem}{Theorem}[section]

\newtheorem{proposition}[theorem]{Proposition}

\theoremstyle{definition}
\newtheorem{definition}[theorem]{Definition}

\numberwithin{equation}{section}
\input{tcilatex}

\begin{document}
\title{Metallic K\"{a}hler and Nearly Metallic K\"{a}hler Manifolds}
\author{Sibel TURANLI}
\address{Erzurum Technical University, Faculty of Science, Department of
Mathematics, Erzurum-TURKEY.}
\email{sibel.turanli@erzurum.edu.tr}
\author{Aydin GEZER}
\address{Ataturk University, Faculty of Science, Department of Mathematics,
25240, Erzurum-TURKEY.}
\email{aydingzr@gmail.com}
\author{Hasan CAKICIOGLU}
\address{Ataturk University, Faculty of Science, Department of Mathematics,
25240, Erzurum-TURKEY.}
\email{h.cakicioglu@gmail.com}

\begin{abstract}
In this paper, we construct metallic K\"{a}hler and nearly metallic K\"{a}%
hler structures on Riemanian manifolds. For such manifolds with these
structures, we study curvature properties. Also we describe linear
connections on the manifold, which preserve the associated fundamental
2-form and satisfy some additional conditions and present some results
concerning them.

\textbf{2010 Mathematics subject classifications:} Primary 53C55; Secondary
53C05.

\textbf{Keywords:} K\"{a}hler structure, Linear connection, Riemannian
curvature tensor.
\end{abstract}

\maketitle

\section{Basic Definitions and Results}

Let $M_{n}$ be an $n-$dimensional manifold. We point out here and once that
all geometric objects considered in this paper are supposed to be of class $%
C^{\infty }$.

The number $\eta =\frac{1+\sqrt{5}}{2}\approx 1,61803398874989...$ , which
is the positive root of the equation $x^{2}-x-1=0$, represents the golden
mean. There are two of the most important generalizations of the golden
mean. The first of them is the golden $p-$proportions being a positive root
of the equation $x^{p+1}-x^{p}-1=0$, $(p=0,1,2,3,...)$ in \cite{Stakhov}.
The other called metallic means family or metallic proportions was
introduced by V. W. de Spinadel in \cite{de Spinadel1,de Spinadel2,de
Spinadel3,de Spinadel4}. For two positive integers $p$ and $q,$ the positive
solution of the equation $x^{2}-px-q=0$ is named members of the metallic
means family. All the members of the metallic means family are positive
quadratic irrational numbers $\sigma _{p,q}=\frac{p+\sqrt{p^{2}+4q}}{2}.$
These numbers $\sigma _{p,q}$ are also called $(p,q)-$metallic numbers. Now,
we consider the equation $x^{2}-px+\frac{3}{2}q=0$, where $p$ and $q$ are
real numbers satisfying $q\geq 0$ and $-\sqrt{6q}<p<\sqrt{6q}$. In the case,
this equation has complex roots as $\sigma _{p,q}^{c}=\frac{p\pm \sqrt{%
p^{2}-6q}}{2}.$ The complex numbers $\sigma _{p,q}^{c}=\frac{p+\sqrt{p^{2}-6q%
}}{2}$ will be called \textit{c}omplex metallic means family by us. In
particular, if $p=1$ and $q=1$, then the complex metallic means family $%
\sigma _{p,q}^{c}=\frac{p+\sqrt{p^{2}-6q}}{2}$ reduces to the complex golden
mean: $\sigma _{1,1}^{c}=\frac{1+\sqrt{5}i}{2},i^{2}=-1$ which is a complex
analog of well-known golden mean \cite{Crasmareanu}. By inspiring from the
complex metallic means family, we will establish a new structure on a
Riemannian manifold and call it an almost complex metallic structure. An
almost complex metallic structure is a $(1,1)-$tensor field $J_{M}$ which
satisfies the relation%
\begin{equation*}
{J_{M}}^{2}-pJ_{M}+\frac{3}{2}qI=0,
\end{equation*}%
where $I$ is the identity operator on the Lie algebra of vector fields on $%
M_{n}$ and $p$, $q$ are real numbers satisfying $q\geq 0$ and $-\sqrt{6q}<p<%
\sqrt{6q}$. Indeed, an almost complex metallic structure is an example of
polynomial structures of degree $2$ which was generally defined by S. I.
Goldberg, K. Yano and N. C. Petridis in (\cite{Goldberg} and \cite{Goldberg1}%
). Throughout this paper, we will sign by $J_{M}$ an almost complex metallic
structure. It is clear that such a structure exists only when $M$ is of even
dimension. Because of this, we will take $n=2k$.

The following result gives relationships between the almost complex
structures and almost complex metallic structures on $M_{2k}$.

\begin{proposition}
\label{pro2.2} If $J_{M}$ is an almost complex metallic structure on $M_{2k}$%
, then 
\begin{equation*}
J_{\pm }=\pm \left( \frac{2}{2\sigma _{p,q}^{c}-p}J_{M}-\frac{2p}{2\sigma
_{p,q}^{c}-p}I\right)
\end{equation*}%
are two almost complex structures on $M_{2k}$. Conversely, if $J$ is an
almost complex structure on $M_{2k}$, then%
\begin{equation*}
J_{M}=\frac{p}{2}I\pm \left( \frac{2\sigma _{p,q}^{c}-p}{2}\right) J
\end{equation*}
are two almost complex metallic structures on $M_{2k}$, where $\sigma
_{p,q}^{c}=\frac{p+\sqrt{p^{2}-6q}}{2}.$
\end{proposition}

\begin{proof}
Let us assume that $J_{M}$ is an almost complex metallic structure on $%
M_{2k} $. Then%
\begin{eqnarray*}
J^{2} &=&{\left( \pm \left( \frac{2}{2\sigma _{p,q}^{c}-p}J_{M}-\frac{2p}{%
2\sigma _{p,q}^{c}-p}I\right) \right) }^{2} \\
&=&\frac{4}{\left\vert p^{2}-6q\right\vert }J_{M}^{2}-\frac{4p}{\left\vert
p^{2}-6q\right\vert }J_{M}+\frac{p^{2}}{\left\vert p^{2}-6q\right\vert }I \\
&=&\frac{1}{\left\vert p^{2}-6q\right\vert }\left( 4\left( pJ_{M}-\frac{3}{2}%
qI\right) -4pJ_{M}+p^{2}I\right) \\
&=&\frac{1}{\left\vert p^{2}-6q\right\vert }\left(
4pJ_{M}-6qI-4pJ_{M}+p^{2}I\right) \\
&=&\frac{p^{2}-6q}{\left\vert p^{2}-6q\right\vert }I \\
&=&-I.
\end{eqnarray*}

In constrast, let $J$ be an almost complex structure on $M_{2k}$. Then 
\begin{eqnarray*}
&&{J_{M}}^{2}-pJ_{M}+\frac{3}{2}qI \\
&=&{\left( \frac{p}{2}I\pm \left( \frac{2{\sigma }_{p,q}^{c}-p}{2}\right)
J\right) }^{2}-p\left( \frac{p}{2}I\pm \left( \frac{2{\sigma }_{p,q}^{c}-p}{2%
}\right) J\mathrm{\ }\right) +\frac{3}{2}qI \\
&=&\frac{p^{2}}{4}I\pm \frac{p\sqrt{p^{2}-6q}}{2}J+\frac{\left\vert
p^{2}-6q\right\vert }{4}J^{2}-\frac{p^{2}}{2}I\mp \frac{p\sqrt{p^{2}-6q}}{2}%
J+\frac{3}{2}I \\
&=&\frac{p^{2}}{4}I+\frac{p^{2}}{4}-\frac{6q}{4}-\frac{p^{2}}{2}I+\frac{3}{2}%
I \\
&=&0.
\end{eqnarray*}%
$\ \ \ \ \ \ \ \ \ \ \ \ \ \ \ \ \ \ \ \ \ \ \ \ \ \ \ \ \ \ \ \ \ \ \ \ \ \
\ \ \ \ \ \ \ \ \ \ \ \ \ \ \ \ \ \ \ \ \ \ \ \ \ \ \ $
\end{proof}

\bigskip

Note that the followings satisfy:

$i)$ if $J$ is an almost complex structure, then $\widehat{J}=-J$ is an
almost complex structure,

$ii)$ if $J_{M}$ is an almost complex metallic structure, then $\widehat{%
J_{M}}=pI-J_{M}$ is an almost complex metallic structure. In fact 
\begin{eqnarray*}
&&\widehat{J_{M}}^{2}-p\widehat{J_{M}}+\frac{3}{2}qI \\
&=&{\left( pI-J_{M}\right) }^{2}-p\left( pI-J_{M}\right) +\frac{3}{2}qI \\
&=&p^{2}I-2pJ_{M}+J_{M}^{2}-p^{2}I+pJ_{M}+\frac{3}{2}qI \\
&=&-2pJ_{M}+pJ_{M}-\frac{3}{2}qJ+pJ_{M}+\frac{3}{2}qI\ \  \\
&=&0.
\end{eqnarray*}%
$\widehat{J}$ and $\widehat{J_{M}}$ are called the conjugate almost complex
structure and the conjugate almost complex metallic structure, respectively.
From Proposition \ref{pro2.2}, it is easy to see that the almost complex
structure $J$ (resp. $\widehat{J})$ defines a $J$ (resp. $\widehat{J}$)$-$%
associated almost complex metallic structure $J_{M}$ (resp. $\widehat{J_{M}}$%
), and vice versa. Hence, there exist an $1:1$ correspondence between almost
complex metallic structures and almost complex structures on $M_{2k}$.

If a manifold $M_{2k}$ has an almost complex metallic structure $J_{M}$,
then the pair $(M_{2k},J_{M})$ is an almost complex metallic manifold.
Recall that a polynomial structure is integrable if the Nijenhuis tensor
vanishes \cite{Vanzura}. Then, the integrability of $J_{M}$ is equivalent to
the vanishing of the Nijenhuis tensor $N_{J_{M}}$:%
\begin{equation*}
N_{J_{M}}(X,Y)=\left[ J_{M}X,J_{M}Y\right] -J_{M}\left[ J_{M}X,Y\right]
-J_{M}\left[ X,J_{M}Y\right] +J_{M}^{2}\left[ X,Y\right] .
\end{equation*}%
If the almost complex metallic structure $J_{M}$ is integrable, \ then this
structure is called a complex metallic structure and the pair $%
(M_{2k},J_{M}) $ is called a complex metallic manifold. A Riemannian metric
on an almost complex metallic manifold $(M_{2k},J_{M})$ is hyperbolic with
respect to $J_{M}$ if it satisfies 
\begin{equation}
g(J_{M}X,Y)=-g(X,J_{M}Y)  \label{TCG0}
\end{equation}%
or equivalently%
\begin{equation}
g\left( J_{M}X,J_{M}Y\right) =-pg\left( X,J_{M}Y\right) +\frac{3}{2}qg\left(
X,Y\right)  \label{TCG1}
\end{equation}%
for any vector fields $X$ and $Y$ on $M_{2k}$. Also we refer to the
conditions (\ref{TCG0}) or (\ref{TCG1}) as the hyperbolic compatibility of $%
g $ and $J_{M}$ and call $g$ hyperbolic metric\textit{. }An almost complex
metallic manifold $(M_{2k},J_{M})$ equipped with a hyperbolic metric $g$ is
called an almost metallic Hermitian manifold.

\begin{proposition}
\label{pro2.3}Let $J$ (resp. $\widehat{J}$) be an almost complex structure
on a Riemannian manifold $(M_{2k},g)$ and $J_{M}$ (resp. $\widehat{J_{M}}$)
be a $J$ (resp. $\widehat{J}$)$-$associated almost complex metallic
structure. The following statements are equivalent:

i) $g$ is hyperbolic with respect to $J$.

ii) $g$ is hyperbolic with respect to $\widehat{J}$.

iii) $g$ is hyperbolic with respect to $J_{M}$.

iv) $g$ is hyperbolic with respect to $\widehat{J_{M}}$.
\end{proposition}

\begin{proof}
We only prove the equivalence of \textit{i)} and \textit{iv)} as the rest of
the cases follow by the similar argument.

Assuming \textit{i)}, then, for all vector fields $X$ and $Y$ on $M_{2k}$%
\begin{eqnarray*}
g\left( \widehat{{J}_{M}}X,Y\right) &=&g\left( \left( \frac{p}{2}I\pm \left( 
\frac{2{\sigma }_{p,q}^{c}-p}{2}\right) \widehat{J}\right) X,Y\right) \\
&=&\frac{p}{2}g\left( X,Y\right) \pm \frac{2{\sigma }_{p,q}^{c}-p}{2}g\left( 
\widehat{J}X,Y\right) \\
&=&\frac{p}{2}g\left( X,Y\right) \mp \frac{2{\sigma }_{p,q}^{c}-p}{2}g\left(
JX,Y\right) \\
&=&\frac{p}{2}g\left( X,Y\right) \pm \frac{2{\sigma }_{p,q}^{c}-p}{2}g\left(
X,JY\right) \\
&=&-g\left( X,\left( \frac{p}{2}I\pm \left( \frac{2{\sigma }_{p,q}-p}{2}%
\right) \widehat{J}\right) Y\right) \\
&=&-g\left( X,\widehat{{J}_{M}}Y\right) .
\end{eqnarray*}

Next assuming \textit{iv)}, then, for all vector fields $X$ and $Y$ on $%
M_{2k}$%
\begin{eqnarray*}
g\left( JX,Y\right) &=&-g\left( \widehat{J}X,Y\right) \\
&=&\mp g\left( \left( \frac{2}{2{\sigma }_{p,q}-p}\widehat{J_{M}}-\frac{2p}{2%
{\sigma }_{p,q}-p}I\right) X,Y\right) \\
&=&\mp \frac{2}{2{\sigma }_{p,q}-p}g\left( {\widehat{J_{M}}}X,Y\right) \pm 
\frac{2p}{2{\sigma }_{p,q}-p}g\left( X,Y\right) \\
&=&\pm \frac{2}{2{\sigma }_{p,q}-p}g\left( X,{\widehat{J_{M}}}Y\right) \pm 
\frac{2p}{2{\sigma }_{p,q}-p}g\left( X,Y\right)
\end{eqnarray*}%
\begin{equation*}
=g\left( X,\pm \left( {\frac{2}{2{\sigma }_{p,q}-p}\widehat{J_{M}}}-\frac{2p%
}{2{\sigma }_{p,q}-p}I\right) Y\right)
\end{equation*}%
\begin{equation*}
=g\left( X,\pm \left( {\frac{2}{2{\sigma }_{p,q}-p}\widehat{J_{M}}}-\frac{2p%
}{2{\sigma }_{p,q}-p}I\right) Y\right)
\end{equation*}%
\begin{equation*}
=g\left( X,\widehat{J}Y\right) =-g\left( X,JY\right) .
\end{equation*}
\end{proof}

From Proposition \ref{pro2.3}, we immediately say that the following
statements are equivalent:

\textit{i) }The triple $(M_{2k},g,J)$ is an almost Hermitian manifold.

\textit{ii)} The triple $(M_{2k},g,\widehat{J})$ is an almost Hermitian
manifold.

\textit{iii)} The triple $(M_{2k},g,J_{M})$ is an almost metallic Hermitian
manifold.

\textit{iv)} The triple $(M_{2k},g,\widehat{J_{M}})$ is an almost metallic
Hermitian manifold.

\section{\protect\bigskip Metallic K\"{a}hler Manifolds}

In the following, let $(M_{2k},g,J_{M})$ be an almost metallic Hermitian
manifold. Here and in the following, let $\nabla $ always denote the
Levi-Civita connection of $g$.

\begin{proposition}
\label{pro2.4}Let $\left( M_{2k},g,J_{M}\right) $ be an almost metallic
Hermitian manifold and $\nabla $ be the Levi-Civita connection of $g$. Then
the following statements hold:

i) $\left( {\nabla }_{X}J_{M}\right) J_{M}Y=\widehat{{J}_{M}}\left( {\nabla }%
_{X}J_{M}\right) Y$

ii) $g\left( \left( {\nabla }_{X}J_{M}\right) Y,Z\right) =-g\left( Y,\left( {%
\nabla }_{X}J_{M}\right) Z\right) $

for all vector fields $X,Y$ and $Z$ on $M_{2k}$, where $\widehat{{J}_{M}}$
is the conjugate almost complex metallic structure.
\end{proposition}

\begin{proof}
\textit{i)} It follows that 
\begin{eqnarray*}
{\nabla }_{X}(J_{M}^{2})Y &=&\left( {\nabla }_{X}J_{M}\right)
J_{M}Y+J_{M}\left( {\nabla }_{X}J_{M}\right) Y \\
{\nabla }_{X}\left( pJ_{M}-\frac{3}{2}qI\right) Y &=&\left( {\nabla }%
_{X}J_{M}\right) J_{M}Y+J_{M}\left( {\nabla }_{X}J_{M}\right) Y \\
p\left( {\nabla }_{X}J_{M}\right) Y &=&\left( {\nabla }_{X}J_{M}\right)
J_{M}Y+J_{M}\left( {\nabla }_{X}J_{M}\right) Y \\
\left( {\nabla }_{X}J_{M}\right) J_{M}Y &=&(pI-J_{M})\left( {\nabla }%
_{X}J_{M}\right) Y \\
\left( {\nabla }_{X}J_{M}\right) J_{M}Y &=&\widehat{J_{M}}\left( {\nabla }%
_{X}J_{M}\right) Y.
\end{eqnarray*}

\textit{ii) }The statement is direct consequence of (\ref{TCG0}) and $\nabla
g=0$.
\end{proof}

Now, we consider the $(0,3)-$tensor field $F$, which will later be used for
characterizing the almost metallic Hermitian manifold. The $(0,3)-$tensor
field $F$ is defined by%
\begin{equation*}
F(X,Y,Z)=g((\nabla _{X}J_{M})Y,Z)
\end{equation*}%
for all vector fields $X,Y$ and $Z$ on $M_{2k}$.

\begin{proposition}
\noindent On an almost metallic Hermitian manifold $\left(
M_{2k},g,J_{M}\right) $, the $(0,3)-$tensor field $F$ satisfies the
following properties:

i) $F\left( X,Y,Z\right) =-F\left( X,Z,Y\right) $

ii) $F\left( X,\ J_{M}Y,J_{M}Z\right) =\frac{3}{2}qF(X,Z,Y)$ \noindent for
all vector fields $X,Y$ and $Z$ on $M_{2k}.$
\end{proposition}

\begin{proof}
\textit{i) }The statement immediately follows\textit{\ }from Proposition \ref%
{pro2.4}.

\textit{ii) }By means of Proposition \ref{pro2.4}, we have%
\begin{eqnarray*}
\mathit{\ }F\left( X,J_{M}Y,J_{M}Z\right) &=&g\left( \left( {\nabla }%
_{X}J_{M}\right) J_{M}Y,J_{M}Z\right) \\
&=&g(\widehat{J_{M}}({\nabla }_{X}J_{M})Y,J_{M}Z) \\
&=&-g(J_{M}\widehat{J_{M}}({\nabla }_{X}J_{M})Y,Z) \\
&=&\frac{3}{2}qg(({\nabla }_{X}J_{M})Z,Y) \\
&=&\frac{3}{2}qF(X,Z,Y).
\end{eqnarray*}
\end{proof}

\noindent

\noindent\ The $2-$covariant skew-symmetric tensor field $\omega $ defined
by $\omega (X,Y)=g(J_{M}X,Y)$ is the fundamental $2-$form of the almost
metallic Hermitian manifold $\left( M_{2k},g,J_{M}\right) .$

\begin{proposition}
\label{pro2.5}Let $\left( M_{2k},g,J_{M}\right) $ be an almost metallic
Hermitian manifold and $\nabla $ be the Levi-Civita connection of $g$. The
following statement holds:%
\begin{equation*}
3qF(X,Y,Z)+g\left( \widehat{{J}_{M}}X,N_{J_{M}}\left( Y,Z\right) \right)
=3d\omega \left( X,J_{M}Y,J_{M}Z\right) -\frac{9}{2}qd\omega \left(
X,Y,Z\right)
\end{equation*}%
for all vector fields $X,Y$ and $Z$ on $M_{2k}$, where $\omega $ is the
fundamental $2-$form and $N_{J_{M}}$ is the Nijenhuis tensor of $J_{M}$.
\end{proposition}

\begin{proof}
By the Cartan's formula, we have%
\begin{equation}
3d\omega \left( X,Y,Z\right) =g\left( Y,\left( {\nabla }_{X}J_{M}\right)
Z\right) +g\left( Z,\left( {\nabla }_{Y}J_{M}\right) X\right) +g\left(
X,\left( {\nabla }_{Z}J_{M}\right) Y\right) .  \label{TCG2}
\end{equation}%
When writing $Y=J_{M}Y$ and $Z=J_{M}Z$ in (\ref{TCG2}), we find\noindent 
\begin{eqnarray}
&&3d\omega \left( X,J_{M}Y,J_{M}Z\right) =g\left( J_{M}Y,\left( {\nabla }%
_{X}J_{M}\right) J_{M}Z\right)  \label{TCG3} \\
&&+g\left( J_{M}Z,\left( {\nabla }_{J_{M}Y}J_{M}\right) X\right) +g\left(
X,\left( {\nabla }_{J_{M}Z}J_{M}\right) J_{M}Y\right) .  \notag
\end{eqnarray}%
Subtracting (\ref{TCG3}) from (\ref{TCG2}), we have%
\begin{eqnarray*}
&&3d\omega \left( X,J_{M}Y,J_{M}Z\right) -\frac{9q}{2}d\omega \left(
X,Y,Z\right) \\
&=&g\left( J_{M}Y,\left( {\nabla }_{X}J_{M}\right) J_{M}Z\right) +g\left(
J_{M}Z,\left( {\nabla }_{J_{M}Y}J_{M}\right) X\right) \\
&&+g\left( X,\left( {\nabla }_{J_{M}Z}J_{M}\right) J_{M}Y\right) -\frac{3q}{2%
}g\left( Y,\left( {\nabla }_{X}J_{M}\right) Z\right) \\
&&-\frac{3q}{2}g\left( Z,\left( {\nabla }_{Y}J_{M}\right) X\right) -\frac{3q%
}{2}g\left( X,\left( {\nabla }_{Z}J_{M}\right) Y\right) \\
&=&-g\left( \left( {\nabla }_{X}J_{M}\right) J_{M}Y,J_{M}Z\right) -g\left(
\left( {\nabla }_{J_{M}Y}J_{M}\right) J_{M}Z,X\right) \\
&&+g\left( X,\left( {\nabla }_{J_{M}Z}J_{M}\right) J_{M}Y\right) +\frac{3q}{2%
}g\left( \left( {\nabla }_{X}J_{M}\right) Y,Z\right) \\
&&+\frac{3q}{2}g\left( \left( {\nabla }_{Y}J_{M}\right) Z,X\right) -\frac{3q%
}{2}g\left( X,\left( {\nabla }_{Z}J_{M}\right) Y\right)
\end{eqnarray*}%
\begin{eqnarray*}
&=&-g\left( \widehat{{J}_{M}}\left( {\nabla }_{X}J_{M}\right)
Y,J_{M}Z\right) -g\left( \widehat{{J}_{M}}\left( {\nabla }%
_{J_{M}Y}J_{M}\right) Z,X\right) \\
&&+g\left( X,\widehat{{J}_{M}}\left( {\nabla }_{J_{M}Z}J_{M}\right) Y\right)
+\frac{3q}{2}g\left( \left( {\nabla }_{X}J_{M}\right) Y,Z\right) \\
&&-g\left( J_{M}\left( {\nabla }_{Y}J_{M}\right) Z,\widehat{{J}_{M}}X\right)
+g\left( \widehat{{J}_{M}}X,J_{M}\left( {\nabla }_{Z}J_{M}\right) Y\right) \\
&=&\frac{3q}{2}g\left( \left( {\nabla }_{X}J_{M}\right) Y,Z\right) +g\left(
\left( {\nabla }_{J_{M}Y}J_{M}\right) Z,\widehat{{J}_{M}}X\right) \\
&&-g\left( \left( {\nabla }_{J_{M}Z}J_{M}\right) Y,\widehat{{J}_{M}}X\right)
+\frac{3q}{2}g\left( \left( {\nabla }_{X}J_{M}\right) Y,Z\right) \\
&&-g\left( J_{M}\left( {\nabla }_{Y}J_{M}\right) Z,\widehat{{J}_{M}}X\right)
+g\left( J_{M}\left( {\nabla }_{Z}J_{M}\right) Y,\widehat{{J}_{M}}X\right) \\
&=&3qg\left( \left( {\nabla }_{X}J_{M}\right) Y,Z\right) +g(\left( {\nabla }%
_{J_{M}Y}J_{M}\right) Z-\left( {\nabla }_{J_{M}Z}J_{M}\right) Y \\
&&+J_{M}\left( {\nabla }_{Z}J_{M}\right) Y-J_{M}\left( {\nabla }%
_{Y}J_{M}\right) Z,\widehat{{J}_{M}}X) \\
&=&3qF(X,Y,Z)+g\left( \widehat{{J}_{M}}X,N_{J_{M}}\left( Y,Z\right) \right) .
\end{eqnarray*}%
Thus, we have our relation.
\end{proof}

\begin{theorem}
\label{pro2.6}Let $\left( M_{2k},g,J_{M}\right) $ be an almost matallic
Hermitian manifold and $\nabla $ be the Levi-Civita connection of $g$. The
conditions $d\omega =0$ and $N_{J_{M}}=0$ are equivalent to $\nabla J_{M}=0.$
\end{theorem}

\begin{proof}
It easy to see that $(\nabla _{X}\omega )(Y,Z)=g((\nabla
_{X}J_{M})Y,Z)=F(X,Y,Z)$ for any vector fields $X,Y,Z$ on $M_{2k}$. Assuming
that $F(X,Y,Z)=0$, i.e., $\nabla J_{M}=0$. Then $d\omega =0$ obviously.
Furthermore, by Proposition \ref{pro2.5}, we obtain $N_{J_{M}}=0$.

Conversely, assuming that $d\omega =0$ and $N_{J_{M}}=0$. The result
immediately follows from by Proposition \ref{pro2.5}.
\end{proof}

If the fundamental $2-$form $\omega $ is closed, i.e., $d\omega =0$, then we
will call the triple $\left( M_{2k},g,J_{M}\right) $ an almost metallic K%
\"{a}hler manifold. Moreover, if $d\omega =0$ and $N_{J_{M}}=0$, we will
call the triple $\left( M_{2k},g,J_{M}\right) $ a metallic K\"{a}hler
manifold. In view of Theorem \ref{pro2.6}, an almost metallic Hermitian
manifold $\left( M_{2k},g,J_{M}\right) $ is a metallic K\"{a}hler manifold
if and only if $\nabla J_{M}=0$.

\subsection{Curvature properties}

Let $(M_{2k},g,J_{M})$ be a metallic K\"{a}hler manifold. Denote by $R$ and $%
S$ the Riemannian curvature tensor and the Ricci tensor of $M_{2k}$,
respectively.

\begin{theorem}
\label{pro2.7}Let $(M_{2k},g,J_{M})$ be a metallic K\"{a}hler manifold. The
following statements hold:

i) $R\left( X,Y\right) J_{M}Z=J_{M}R\left( X,Y\right) Z$ and $R\left(
J_{M}X,J_{M}Y\right) Z=-pR\left( J_{M}X,Y\right) Z+\frac{3q}{2}R\left(
X,Y\right) Z$ for all vector fields $X,Y,Z$ on $M_{2k}.$

ii) $S\left( J_{M}X,J_{M}Y\right) =\left( p^{2}-\frac{9q^{2}p^{2}}{4}+\frac{%
9q^{2}}{4}\right) S\left( X,Y\right) +\left( \frac{3pq}{2}-\frac{9q^{2}p}{4}%
\right) S\left( X,J_{M}Y\right) $ and $\left( 1+\frac{3q}{2}\right) \
S\left( X,Y\right) -p\ S\left( X,J_{M}Y\right) =-\frac{2}{3q}trace\widehat{{J%
}_{M}}R\left( X,J_{M}Y\right) $ for all vector fields $X,Y$ on $M_{2k}$.
\end{theorem}

\begin{proof}
\textit{i)} By applying the Ricci identity to $J_{M}$, the first relation
immediately follows from $\nabla J_{M}=0$. For any vector fields $X,Y,Z$ and 
$W$ on $M_{2k}$, we get 
\begin{eqnarray*}
&&g\left( R\left( J_{M}X,J_{M}Y\right) Z,W\right) \\
&=&R\left( J_{M}X,J_{M}Y,Z,W\right) =R\left( {Z,W,J}_{M}X,J_{M}Y\right) \\
&=&R\left( {W,Z,J}_{M}Y,J_{M}X\right) =g\left( R\left( W,Z\right)
J_{M}Y,J_{M}X\right) \\
&=&g\left( J_{M}R\left( W,Z\right) Y,J_{M}X\right) =-pR\left(
W,Z,Y,J_{M}X\right) +\frac{3q}{2}R\left( W,Z,Y,X\right) \\
&=&-pR\left( J_{M}X,Y,Z,W\right) +\frac{3q}{2}R\left( X,Y,Z,W\right) \\
&=&-pg\left( R\left( J_{M}X,Y\right) Z,W\right) +\frac{3q}{2}g\left( R\left(
X,Y\right) Z,W\right)
\end{eqnarray*}%
from which we have 
\begin{equation*}
R\left( J_{M}X,J_{M}Y\right) Z=-pR\left( J_{M}X,Y\right) Z+\frac{3q}{2}%
R\left( X,Y\right) Z.
\end{equation*}

\textit{ii)} Let $\{e_{1},e_{2},...,e_{2k}\}$ be an orthonormal basis of $%
M_{2k}$. For any vector fields $X,Y$ on $M_{2k}$, we have 
\begin{eqnarray}
&&S\left( J_{M}X,J_{M}Y\right)  \label{TCG4} \\
&=&\sum {g\left( R\left( e_{i},J_{M}X\right) J_{M}Y,e_{i}\right) }  \notag \\
&=&\sum {g\left( R\left( {J_{M}e}_{i},J_{M}X\right) J_{M}Y,{J_{M}e}%
_{i}\right) }  \notag \\
&=&\sum {g\left( J_{M}R\left( {J_{M}e}_{i},J_{M}X\right) Y,{J_{M}e}%
_{i}\right) }  \notag \\
&{=}&-\sum {g\left( R\left( {J_{M}e}_{i},J_{M}X\right) Y,{J_{M}^{2}e}%
_{i}\right) }  \notag \\
&=&-P\sum {g\left( R\left( {J_{M}e}_{i},J_{M}X\right) Y,{J_{M}e}_{i}\right) +%
\frac{3q}{2}}\sum {g\left( R\left( {J_{M}e}_{i},J_{M}X\right) Y,e_{i}\right) 
}  \notag \\
&=&-p\sum {R\left( {J_{M}e}_{i},J_{M}X,Y,{J_{M}e}_{i}\right) }+\frac{3q}{2}%
\sum {R\left( {J_{M}e}_{i},J_{M}X,Y,e_{i}\right) }  \notag \\
&=&p^{2}\sum {R\left( {J_{M}e}_{i},X,Y,{J_{M}e}_{i}\right) }-\frac{3pq}{2}%
\sum {R\left( e_{i},X,Y,{J_{M}e}_{i}\right) }  \notag \\
&&-\frac{3pq}{2}\sum {R\left( {J_{M}e}_{i},X,Y,e_{i}\right) }+\frac{9q^{2}}{4%
}\sum {R\left( e_{i},X,Y,e_{i}\right) .}  \notag
\end{eqnarray}%
Also we yield%
\begin{eqnarray}
&&-\frac{3pq}{2}\sum {R\left( e_{i},X,Y,{J_{M}e}_{i}\right) }  \label{TCG6}
\\
&=&-\frac{3pq}{2}\sum {g\left( R\left( e_{i},X\right) Y,{J_{M}e}_{i}\right) }
\notag \\
&=&\frac{3pq}{2}\sum {g\left( J_{M}R\left( e_{i},X\right) Y,e_{i}\right) } 
\notag \\
&=&\frac{3pq}{2}\sum {g\left( R\left( e_{i},X\right) J_{M}Y,e_{i}\right) } 
\notag
\end{eqnarray}%
and 
\begin{eqnarray}
&&-\frac{3pq}{2}\sum {R\left( J_{M}e_{i},X,Y,e_{i}\right) }  \label{TCG7} \\
&=&-\frac{3pq}{2}\sum {R\left( J_{M}\widehat{{J}_{M}}e_{i},X,Y,\widehat{{J}%
_{M}}e_{i}\right) }  \notag \\
&=&-\frac{3pq}{2}\sum {R\left( \frac{3q}{2}e_{i},X,Y,\left( pI-J_{M}\right)
e_{i}\right) }  \notag \\
&=&-\frac{9p^{2}q^{2}}{4}\sum {R\left( e_{i},X,Y,e_{i}\right) }+\frac{9pq^{2}%
}{4}\sum {R\left( e_{i},X,Y,{J_{M}e}_{i}\right) }  \notag \\
&=&-\frac{9p^{2}q^{2}}{4}\sum {g\left( R\left( e_{i},X\right) Y,e_{i}\right) 
}+\frac{9pq^{2}}{4}\sum {g\left( R\left( e_{i},X\right) Y,{J_{M}e}%
_{i}\right) }  \notag \\
&=&-\frac{9p^{2}q^{2}}{4}\sum {g\left( R\left( e_{i},X\right) Y,e_{i}\right) 
}-\frac{9pq^{2}}{4}\sum {g\left( J_{M}R\left( e_{i},X\right) Y,e_{i}\right) }
\notag \\
&=&-\frac{9p^{2}q^{2}}{4}\sum {g\left( R\left( e_{i},X\right) Y,e_{i}\right) 
}-\frac{9pq^{2}}{4}\sum {g\left( R\left( e_{i},X\right) J_{M}Y,e_{i}\right) .%
}  \notag
\end{eqnarray}%
Substituting (\ref{TCG6}) and (\ref{TCG7}) into (\ref{TCG4}), we get 
\begin{eqnarray*}
&&S\left( J_{M}X,J_{M}Y\right) \\
&=&p^{2}\sum {R\left( {J_{M}e}_{i},X,Y,{J_{M}e}_{i}\right) }+\frac{3pq}{2}%
\sum {g\left( R\left( e_{i},X\right) J_{M}Y,e_{i}\right) } \\
&&-\frac{9p^{2}q^{2}}{4}\sum {g\left( R\left( e_{i},X\right) Y,e_{i}\right) }%
-\frac{9pq^{2}}{4}\sum {g\left( R\left( e_{i},X\right) J_{M}Y,e_{i}\right) }
\\
&&{+\frac{9q^{2}}{4}\sum {R\left( e_{i},X,Y,e_{i}\right) }} \\
&=&p^{2}\sum {g\left( R\left( {J_{M}e}_{i},X\right) Y,{J_{M}e}_{i}\right) +}%
\frac{3pq}{2}\sum {g\left( R\left( e_{i},X\right) J_{M}Y,e_{i}\right) } \\
&&{-\frac{9p^{2}q^{2}}{4}\sum {g\left( R\left( e_{i},X\right) Y,e_{i}\right) 
}}-\frac{9pq^{2}}{4}\sum {g\left( R\left( e_{i},X\right) J_{M}Y,e_{i}\right) 
} \\
&&{+\frac{9q^{2}}{4}\sum {g\left( R\left( e_{i},X\right) Y,e_{i}\right) }} \\
&=&p^{2}S\left( X,Y\right) +\frac{3pq}{2}S\left( X,J_{M}Y\right) -\frac{%
9p^{2}q^{2}}{4}S\left( X,Y\right) \\
&&-\frac{9pq^{2}}{4}S\left( X,J_{M}Y\right) +\frac{9q^{2}}{4}S\left(
X,Y\right) \\
&=&\left( p^{2}-\frac{9p^{2}q^{2}}{4}+\frac{9q^{2}}{4}\right) S\left(
X,Y\right) +\left( \frac{3pq}{2}-\frac{9pq^{2}}{4}\right) S\left(
X,J_{M}Y\right) .
\end{eqnarray*}%
Thus, we completes the proof of the first formula of \textit{ii).}

With the help of the first Bianchi's identity, we have 
\begin{eqnarray*}
&&S\left( X,Y\right) \\
&=&\sum {g\left( R\left( e_{i},X\right) Y,e_{i}\right) } \\
&=&\frac{2}{3q}\sum {g\left( \widehat{{J}_{M}}J_{M}R\left( e_{i},X\right)
Y,e_{i}\right) }=\frac{2}{3q}\sum {g\left( \widehat{{J}_{M}}R\left(
e_{i},X\right) J_{M}Y,e_{i}\right) } \\
&=&-\frac{2}{3q}\sum {g\left( \widehat{{J}_{M}}R\left( X,J_{M}Y\right)
e_{i},e_{i}\right) }-\frac{2}{3q}\sum {g\left( \widehat{{J}_{M}}R\left(
J_{M}Y,e_{i}\right) X,e_{i}\right) } \\
&=&-\frac{2}{3q}\sum {g\left( \widehat{{J}_{M}}R\left( X,J_{M}Y\right)
e_{i},e_{i}\right) }-\frac{2}{3q}\sum {g\left( \widehat{{J}_{M}}R\left(
J_{M}Y,{J_{M}e}_{i}\right) X,{J_{M}e}_{i}\right) } \\
&=&-\frac{2}{3q}\sum {g\left( \widehat{{J}_{M}}R\left( X,J_{M}Y\right)
e_{i},e_{i}\right) }+\frac{2}{3q}\sum {g\left( \widehat{{J}_{M}}J_{M}R\left(
J_{M}Y,{J_{M}e}_{i}\right) X,e_{i}\right) } \\
&=&-\frac{2}{3q}\sum {g\left( \widehat{{J}_{M}}R\left( X,J_{M}Y\right)
e_{i},e_{i}\right) }+\sum {g\left( R\left( J_{M}Y,{J_{M}e}_{i}\right)
X,e_{i}\right) } \\
&=&-\frac{2}{3q}\sum {g\left( \widehat{{J}_{M}}R\left( X,J_{M}Y\right)
e_{i},e_{i}\right) }-p\sum {g\left( R\left( J_{M}Y,e_{i}\right)
X,e_{i}\right) } \\
&&+\frac{3q}{2}\sum {g\left( R\left( Y,e_{i}\right) X,e_{i}\right) } \\
&=&-\frac{2}{3q}Trace\ \widehat{{J}_{M}}R\left( X,J_{M}Y\right) -pS\left(
X,J_{M}Y\right) +\frac{3q}{2}S\left( X,J_{M}Y\right)
\end{eqnarray*}%
which completes the proof.
\end{proof}

\begin{theorem}
Let $(M_{2k},g,J_{M})$ be a metallic K\"{a}hler manifold. The Ricci tensor $%
S $ of $M_{2k}$ satisfies 
\begin{eqnarray*}
&&\left( 1+\frac{3q}{2}\right) \left( {\nabla }_{Z}S\right) \left(
X,Y\right) -P\left( {\nabla }_{Z}S\right) \left( X,J_{M}Y\right) \\
&=&\left( 1+\frac{3q}{2}\right) \left( {\nabla }_{X}S\right) \left(
Z,Y\right) -P\left( {\nabla }_{X}S\right) \left( Z,J_{M}Y\right) \\
&&+\left( \frac{2}{3q}+1\right) \left( {\nabla }_{J_{M}Y}S\right) \left( X,%
\widehat{{J}_{M}}Z\right) -P\left( {\nabla }_{J_{M}Y}S\right) \left( X,%
\widehat{{J}_{M}}Z\right)
\end{eqnarray*}%
for all vector fields $X,Y,Z$ on $M_{2k}$.
\end{theorem}

\begin{proof}
From the second relation of \textit{ii)} in Theorem \ref{pro2.7} and the
second Bianchi's identity we have 
\begin{eqnarray}
&&\left( 1+\frac{3q}{2}\right) \left( {\nabla }_{Z}S\right) \left(
X,Y\right) -P\left( {\nabla }_{Z}S\right) \left( X,J_{M}Y\right)
\label{TCG8} \\
&=&-\frac{2}{3q}\sum {g\left( \widehat{{J}_{M}}\left( {\nabla }_{Z}R\right)
\left( X,J_{M}Y\right) e_{i},e_{i}\right) }  \notag \\
&=&-\frac{2}{3q}\sum {g\left( \widehat{{J}_{M}}\left( {\nabla }_{X}R\right)
\left( Z,J_{M}Y\right) e_{i},e_{i}\right) }  \notag \\
&&-\frac{2}{3q}\sum {g\left( \widehat{{J}_{M}}\left( {\nabla }%
_{J_{M}Y}R\right) \left( X,Z\right) e_{i},e_{i}\right) }  \notag \\
&=&\left( 1+\frac{3q}{2}\right) \left( {\nabla }_{X}S\right) \left(
Z,Y\right) -P\left( {\nabla }_{X}S\right) \left( Z,J_{M}Y\right)  \notag \\
&&-\frac{2}{3q}\sum {g\left( \widehat{{J}_{M}}\left( {\nabla }%
_{J_{M}Y}R\right) \left( X,Z\right) e_{i},e_{i}\right) .}  \notag
\end{eqnarray}%
When writing $Z=J_{M}Z$ ve $Y=\widehat{{J}_{M}}Y$ in the second relation of 
\textit{ii)} in Theorem \ref{pro2.7}, we find 
\begin{eqnarray*}
&&\left( 1+\frac{3q}{2}\right) \left( {\nabla }_{J_{M}Z\mathrm{\ }}S\right)
\left( X,\widehat{{J}_{M}}Y\right) -P\left( {\nabla }_{J_{M}Z\mathrm{\ }%
}S\right) \left( X,J_{M}\widehat{{J}_{M}}Y\right) \\
&=&-\frac{2}{3q}\sum {g\left( \widehat{{J}_{M}}\left( {\nabla }_{J_{M}Z%
\mathrm{\ }}R\right) \left( X,J_{M}\widehat{{J}_{M}}Y\right)
e_{i},e_{i}\right) }
\end{eqnarray*}%
\begin{eqnarray*}
&&\left( 1+\frac{3q}{2}\right) \left( {\nabla }_{J_{M}Z\mathrm{\ }}S\right)
\left( X,\widehat{{J}_{M}}Y\right) -\frac{3pq}{2}\left( {\nabla }_{J_{M}Z%
\mathrm{\ }}S\right) \left( X,Y\right) \\
&=&-\sum {g\left( \widehat{{J}_{M}}\left( {\nabla }_{J_{M}Z\mathrm{\ }%
}R\right) \left( X,Y\right) e_{i},e_{i}\right) }
\end{eqnarray*}%
from which it follows that 
\begin{eqnarray*}
&&-\frac{2}{3q}\sum {g\left( \widehat{{J}_{M}}\left( {\nabla }_{J_{M}Y%
\mathrm{\ }}R\right) \left( X,Z\right) e_{i},e_{i}\right) } \\
&=&\left( \frac{2}{3q}+1\right) \left( {\nabla }_{J_{M}Y\mathrm{\ }}S\right)
\left( X,\widehat{{J}_{M}}Z\right) -p\left( {\nabla }_{J_{M}Y\mathrm{\ }%
}S\right) \left( X,Z\right) .
\end{eqnarray*}%
Substituting the last relation into (\ref{TCG8}), the result follows.
\end{proof}

\section{Nearly metallic K\"{a}hler Manifolds}

Let $(M_{2k},g,J_{M})$ be an almost metallic Hermitian manifold. Following
terminologies used in \cite{Yano} for the almost Hermitian manifolds, we can
say that for a given almost metallic Hermitian manifold $(M_{2k},g,J_{M})$,
if the the fundamental $2-$form $\omega $ satisfies the following relation:%
\begin{equation}
(\nabla _{X}\omega )(Y,Z)+(\nabla _{Y}\omega )(X,Z)=0\,  \label{TCG9}
\end{equation}%
for all vector fields $X,Y$ and $Z$, then we will call the triple $%
(M_{2k},g,J_{M})$ a nearly metallic K\"{a}hler manifold. It is clear that
the relation (\ref{TCG9}) is equivalent to

\begin{equation}
(\nabla _{X}J_{M})Y+(\nabla _{Y}J_{M})X=0\,.  \label{TCG10}
\end{equation}

Next we will prove the following two propositions.

\begin{proposition}
On a nearly metallic K\"{a}hler manifold $(M_{2k},g,J_{M})$, the $(0,3)-$%
tensor field $F$ satisfies the following properties:

i) $F\left( J_{M}X,Y,J_{M}Z\right) =\frac{3q}{2}F(Y,X,Z)$

ii) $F\left( J_{M}X,\ J_{M}Y,Z\right) =-pF(Y,X,\widehat{J_{M}}Z)+\frac{3q}{2}%
F(Y,X,Z)$ \noindent for all vector fields $X,Y$ and $Z$ on $M_{2k}.$
\end{proposition}

\begin{proof}
\textit{i) }It follows that 
\begin{eqnarray*}
F\left( J_{M}X,Y,J_{M}Z\right) &=&g\left( \left( {\nabla }%
_{J_{M}X}J_{M}\right) Y,J_{M}Z\right) \\
&=&-g\left( \left( {\nabla }_{Y}J_{M}\right) J_{M}X,J_{M}Z\right) \\
&=&-g\left( \widehat{J_{M}}\left( {\nabla }_{Y}J_{M}\right) X,J_{M}Z\right)
\\
&=&g(J_{M}\widehat{J_{M}}\left( {\nabla }_{Y}J_{M}\right) X,Z) \\
&=&\frac{3q}{2}g(\left( {\nabla }_{Y}J_{M}\right) X,Z) \\
&=&\frac{3q}{2}F(Y,X,Z)
\end{eqnarray*}%
\textit{ii) }We calculate 
\begin{eqnarray*}
F\left( J_{M}X,J_{M}Y,Z\right) &=&g\left( \left( {\nabla }%
_{J_{M}X}J_{M}\right) J_{M}Y,Z\right) \\
&=&g\left( \widehat{J_{M}}\left( {\nabla }_{J_{M}X}J_{M}\right) Y,Z\right) \\
&=&g\left( \widehat{J_{M}}\left( {\nabla }_{Y}J_{M}\right) X,\widehat{J_{M}}%
Z\right) \\
&=&-pg(\left( {\nabla }_{Y}J_{M}\right) X,\widehat{J_{M}}Z)+\frac{3q}{2}%
g(\left( {\nabla }_{Y}J_{M}\right) X,Z) \\
&=&-pF(Y,X,\widehat{J_{M}}Z)+\frac{3q}{2}F(Y,X,Z).
\end{eqnarray*}
\end{proof}

\begin{theorem}
A nearly metallic K\"{a}hler manifold is integrable if and only if it is a
metallic K\"{a}hler manifold.
\end{theorem}

\begin{proof}
On a nearly metallic K\"{a}hler manifold $(M_{2k},g,J_{M})$, the Nijenhuis
tensor of $J_{M}$ verifies 
\begin{eqnarray*}
N_{J_{M}}\left( X,Y\right) &=&\left[ J_{M}X,J_{M}Y\right] -J_{M}\left[
J_{M}X,Y\right] -J_{M}\left[ X,J_{M}Y\right] +J_{M}^{2}\left[ X,Y\right] \\
&=&\left( {\nabla }_{J_{M}X}J_{M}\right) Y-\left( {\nabla }%
_{J_{M}Y}J_{M}\right) X-J_{M}({\nabla }_{X}J_{M})Y+J_{M}({\nabla }_{Y}J_{M})X
\\
&=&-\left( {\nabla }_{Y}J_{M}\right) J_{M}X+\left( {\nabla }_{X}J_{M}\right)
J_{M}Y-J_{M}({\nabla }_{X}J_{M})Y-J_{M}({\nabla }_{X}J_{M})Y \\
&=&-\widehat{J_{M}}\left( {\nabla }_{Y}J_{M}\right) X+\widehat{J_{M}}\left( {%
\nabla }_{X}J_{M}\right) Y-2J_{M}({\nabla }_{X}J_{M})Y \\
&=&2\widehat{J_{M}}\left( {\nabla }_{X}J_{M}\right) Y-2J_{M}({\nabla }%
_{X}J_{M})Y \\
&=&2(pI-2J_{M})\left( {\nabla }_{X}J_{M}\right) Y
\end{eqnarray*}%
from which we say that $N_{J_{M}}=0$ if and only if $\nabla J_{M}=0$. This
expression completes the proof.
\end{proof}

\subsection{Curvature properties}

Coordinate systems in a nearly metallic K\"{a}hler manifold $\left(
M_{2k},g,J_{M}\right) $ are denoted by $(U,x^{i})$, where $U$ is the
coordinate neighbourhood \ and $x^{i}$, $i=1,2,...,2k$ are the coordinate
functions. Substituting $X=\frac{\partial }{\partial x^{i}}$ and $Y=\frac{%
\partial }{\partial x^{j}}$ in (\ref{TCG9}) and (\ref{TCG10}), one
respectively has%
\begin{equation*}
{\nabla }_{i}\omega _{jm}+{\nabla }_{j}\omega _{im}=0
\end{equation*}%
and

\begin{equation*}
{\nabla }_{i}({J_{M})}_{j}^{h}+{\nabla }_{j}({J_{M})}_{i}^{h}=0.
\end{equation*}%
Contraction with respect to $i$ and $h$ in the last relation, we get $\nabla
_{i}({J_{M}})_{j}^{i}=0$.

\begin{theorem}
\label{Theo1}\noindent The Ricci and Ricci* curvature tensors \textit{in a
nearly metallic }K\"{a}hler manifold $(M_{2k},g,J_{M})$ satisfy $S_{jt}({%
J_{M})}_{i}^{t}=-\frac{2}{3q}S_{jt}^{\ast }({\widehat{J_{M}})}_{i}^{t}$ if
and only if 
\begin{equation*}
{\nabla }^{m}{\nabla }_{j}{\omega }_{im}=0,
\end{equation*}%
where ${\omega }_{im}$ are the components of the fundamental $2-$form $%
\omega $.
\end{theorem}

\begin{proof}
\noindent When applied the Ricci identity to $({J_{M})}_{i}^{h}$, one has

\begin{equation*}
{\nabla }_{k}{\nabla }_{j}({J_{M})}_{i}^{h}-{\nabla }_{j}{\nabla }_{k}({%
J_{M})}_{i}^{h}=R_{kjt}^{\text{ \ \ }h}({J_{M})}_{i}^{t}-R_{kji}^{\text{ \ \ 
}t}({J_{M})}_{t}^{h},
\end{equation*}%
where $R_{kjt}^{\text{ \ \ }h}$ are components of the Riemannian curvature
tensor $R$. Contraction the above relation with respect to $k$ and $h$ gives 
\begin{equation*}
{\nabla }_{h}{\nabla }_{j}({J_{M})}_{i}^{h}-{\nabla }_{j}{\nabla }_{h}({%
J_{M})}_{i}^{h}=R_{hjt}^{\text{ \ \ }h}({J_{M})}_{i}^{t}-R_{hji}^{\text{ \ \ 
}t}({J_{M})}_{t}^{h}
\end{equation*}%
\begin{eqnarray}
{\nabla }_{h}{\nabla }_{j}({J_{M})}_{i}^{h} &=&S_{jt}({J_{M})}%
_{i}^{t}-R_{hji}^{\text{ \ \ }t}({J_{M})}_{t}^{h}  \label{TCG11} \\
&=&S_{jt}({J_{M})}_{i}^{t}-R_{hjil}g^{lt}({J_{M})}_{t}^{h}  \notag \\
&=&S_{jt}({J_{M})}_{i}^{t}-R_{hjil}\omega ^{hl}=S_{jt}({J_{M})}%
_{i}^{t}-H_{ji}.  \notag
\end{eqnarray}%
Here $S_{jt}$ are the components of the Ricci curvature tensor and $\omega
^{hl}$ are the contravariant components of the fundamental $2-$form $\omega $%
. Also note that the tensor $H_{ji}$ is anti-symmetric. In fact 
\begin{equation*}
H_{ji}=R_{hjil}\omega ^{hl}=\frac{1}{2}\left( R_{hjil}+R_{hjil}\right)
\omega ^{hl}=\frac{1}{2}\left( R_{hjil}-R_{ljih}\right) \omega ^{hl}
\end{equation*}%
and similarly 
\begin{equation*}
H_{ij}=R_{hijl}\omega ^{hl}=\frac{1}{2}\left( R_{hijl}+R_{hijl}\right)
\omega ^{hl}=\frac{1}{2}\left( R_{hijl}-R_{lijh}\right) \omega ^{hl}
\end{equation*}%
The sum of the above relations gives 
\begin{equation*}
H_{ij}+H_{ji}=\frac{1}{2}\left( R_{hjil}-R_{ljih}+R_{hijl}-R_{lijh}\right)
\omega ^{hl}=0.
\end{equation*}

The tensor $S^{\ast }$ given by \cite{Yano}%
\begin{equation*}
S_{ji}^{\ast }=-H_{jt}({J_{M})}_{i}^{t}
\end{equation*}%
is called the Ricci* curvature tensor of $M_{2k}$. It is easy to see that 
\begin{equation}
S_{jt}^{\ast }({\widehat{J_{M}})}\text{ }_{i}^{t}=-\frac{3}{2}qH_{ji}.
\label{TCG12}
\end{equation}%
From (\ref{TCG11}) and (\ref{TCG12}) we obtain 
\begin{eqnarray*}
{\nabla }_{t}{\nabla }_{j}({J_{M})}_{i}^{t} &=&S_{jt}({J_{M})}_{i}^{t}+\frac{%
2}{3q}S_{jt}^{\ast }({\widehat{J_{M}})}\text{ }_{i}^{t} \\
{\nabla }_{t}{\nabla }_{j}(g^{mt}{\omega }_{im}) &=&S_{jt}({J_{M})}_{i}^{t}+%
\frac{2}{3q}S_{jt}^{\ast }({\widehat{J_{M}})}\text{ }_{i}^{t} \\
g^{mt}{\nabla }_{t}{\nabla }_{j}{\omega }_{im} &=&S_{jt}({J_{M})}_{i}^{t}+%
\frac{2}{3q}S_{jt}^{\ast }({\widehat{J_{M}})}\text{ }_{i}^{t} \\
{\nabla }^{m}{\nabla }_{j}{\omega }_{im} &=&S_{jt}({J_{M})}_{i}^{t}+\frac{2}{%
3q}S_{jt}^{\ast }({\widehat{J_{M}})}\text{ }_{i}^{t}
\end{eqnarray*}%
which finishes the proof.\noindent
\end{proof}

\noindent

\begin{theorem}
In a nearly metallic K\"{a}hler manifold $(M_{2k},g,J_{M})$, the Ricci
tensor $S$ is hyperbolic with respect to the almost complex metallic
structure $J_{M}$.
\end{theorem}

\begin{proof}
Since the tensor $H$ is an anti-symmetric, we have 
\begin{eqnarray*}
H_{ij}+H_{ji} &=&S_{it}({J_{M})}_{j}^{t}+S_{jt}({J_{M})}_{i}^{t}-\left( {%
\nabla }_{h}{\nabla }_{i}({J_{M})}_{j}^{h}+{\nabla }_{h}{\nabla }_{j}({J_{M})%
}_{i}^{h}\right) \\
0 &=&S_{it}{J_{M}}_{j}^{t}+S_{jt}{J_{M}}_{i}^{t}-{\nabla }_{h}\left( {\nabla 
}_{i}({J_{M})}_{j}^{h}+{\nabla }_{j}({J_{M})}_{i}^{h}\right) \\
S_{ti}({J_{M})}_{j}^{t} &=&-S_{jt}({J_{M})}_{i}^{t}.
\end{eqnarray*}
\end{proof}

\begin{theorem}
In a nearly metallic K\"{a}hler manifold $(M_{2k},g,J_{M})$\textit{, the
Ricci* tensor }$S^{\ast }$\textit{\ is } hyperbolic with respect to the
conjugate almost complex metallic structure $\widehat{J_{M}}$.
\end{theorem}

\begin{proof}
\noindent For the Ricci* curvature tensor $S^{\ast }$ in a nearly metallic K%
\"{a}hler manifold $(M_{2k},g,J_{M})$, with the help of $\omega
^{lh}=-\omega ^{hl}$ and the properties of Riemannian curvature tensor, we
have%
\begin{equation*}
\frac{2}{3q}S_{jm}^{\ast }({\widehat{J_{M}})}\text{ }_{i}^{m}=-H_{ji}
\end{equation*}%
\begin{equation*}
\frac{2}{3q}S_{jm}^{\ast }({\widehat{J_{M}})}\text{ }_{i}^{m}=-R_{hjil}%
\omega ^{lh}
\end{equation*}%
\begin{equation*}
\frac{2}{3q}S_{jm}^{\ast }({\widehat{J_{M}})}\text{ }_{i}^{m}=-\frac{1}{2}%
\left( R_{hjil}+R_{hjil}\right) \omega ^{lh}
\end{equation*}%
\begin{equation}
\frac{2}{3q}S_{jm}^{\ast }({\widehat{J_{M}})}\text{ }_{i}^{m}=-\frac{1}{2}%
\left( R_{hjil}-R_{ljih}\right) \omega ^{lh}  \label{TCG13}
\end{equation}
and similarly%
\begin{equation*}
\frac{2}{3q}S_{im}^{\ast }({\widehat{J_{M}})}\text{ }_{j}^{m}=-H_{ij}
\end{equation*}%
\begin{equation*}
\frac{2}{3q}S_{im}^{\ast }({\widehat{J_{M}})}\text{ }_{j}^{m}=-R_{hijl}%
\omega ^{lh}
\end{equation*}%
\begin{equation}
\frac{2}{3q}S_{im}^{\ast }({\widehat{J_{M}})}\text{ }_{j}^{m}=-\frac{1}{2}%
\left( R_{hjil}-R_{lijh}\right) \omega ^{lh}  \label{TCG14}
\end{equation}

The sum of (\ref{TCG13}) and (\ref{TCG14}) gives%
\begin{equation*}
\frac{2}{3q}\left( S_{jm}^{\ast }({\widehat{J_{M}})}\text{ }%
_{i}^{m}+S_{im}^{\ast }({\widehat{J_{M}})}\text{ }_{j}^{m}\right) =-\frac{1}{%
2}\left( R_{hjil}-R_{ljih}+R_{hjil}-R_{lijh}\right) \omega ^{lh}
\end{equation*}%
\begin{equation*}
S_{jm}^{\ast }({\widehat{J_{M}})}\text{ }_{i}^{m}+S_{im}^{\ast }({\widehat{%
J_{M}})}\text{ }_{j}^{m}=0
\end{equation*}%
\begin{equation*}
S_{jm}^{\ast }({\widehat{J_{M}})}\text{ }_{i}^{m}=-S_{im}^{\ast }({\widehat{%
J_{M}})}\text{ }_{j}^{m}.
\end{equation*}%
Since $S_{im}^{\ast }$ is symmetric, consequently%
\begin{equation*}
S_{jm}^{\ast }({\widehat{J_{M}})}\text{ }_{i}^{m}=-S_{mi}^{\ast }({\widehat{%
J_{M}})}\text{ }_{j}^{m}.
\end{equation*}
\end{proof}

\begin{theorem}
\noindent In a nearly metallic K\"{a}hler manifold $(M_{2k},g,J_{M})$\textit{%
,} the relationship between the scalar and scalar* curvature is as follows:%
\begin{equation*}
S_{c}^{\ast }=\frac{3}{2}qS_{c}+pS_{jt}\omega ^{jt}-{\left\Vert \nabla
J_{M}\right\Vert }^{2},
\end{equation*}%
where $\omega ^{jt}$ are the covariant components of the fundamental $2-$%
form $\omega $.
\end{theorem}

\begin{proof}
In a nearly metallic K\"{a}hler manifold $(M_{2k},g,J_{M})$, transvecting ${%
\nabla }_{j}\omega _{im}=-{\nabla }_{j}\omega _{mi}={\nabla }_{m}\omega
_{ji} $ with $\omega ^{ji}$, it follows that

\begin{equation*}
\left( {\nabla }_{j}\omega _{im}\right) \omega ^{ji}=0.
\end{equation*}%
Taking covariant derivative ${\nabla }_{k}$ of the last relation, we find%
\begin{equation*}
{\nabla }_{k}\{\left( {\nabla }_{j}\omega _{im}\right) \omega ^{ji}\}=0
\end{equation*}%
\begin{equation*}
\left( {{\nabla }_{k}\nabla }_{j}\omega _{im}\right) \omega ^{ji}+\left( {%
\nabla }_{j}\omega _{im}\right) \left( {\nabla }_{k}\omega ^{ji}\right) =0
\end{equation*}%
\begin{equation}
\left( {{\nabla }_{k}\nabla }_{m}\omega _{ji}\right) \omega ^{ji}+\left( {%
\nabla }_{m}\omega _{ji}\right) \left( {\nabla }_{k}\omega ^{ji}\right) =0
\label{TCG15}
\end{equation}%
Transvecting (\ref{TCG15}) by $g^{km}$, we find 
\begin{equation*}
g^{km}\left( {{\nabla }_{k}\nabla }_{m}\omega _{ji}\right) \omega
^{ji}+g^{km}\left( {\nabla }_{m}\omega _{ji}\right) \left( {\nabla }%
_{k}\omega ^{ji}\right) =0
\end{equation*}%
\begin{equation*}
\left( {\nabla }^{m}{\nabla }_{m}\omega _{ji}\right) \omega
^{ji}+g^{km}\left( {\nabla }_{m}g_{jt}({J_{M})}_{i}^{t}\right) \left( {%
\nabla }_{k}g^{is}({J_{M})}_{s}^{j}\right) =0
\end{equation*}%
\begin{equation*}
\left( {\nabla }^{m}{\nabla }_{m}\omega _{ji}\right) \omega
^{ji}+g^{km}g_{jt}g^{is}\left( {\nabla }_{m}({J_{M})}_{i}^{t}\right) \left( {%
\nabla }_{k}({J_{M})}_{s}^{j}\right) =0
\end{equation*}%
\begin{equation*}
\left( S_{jt}({J_{M})}_{i}^{t}+\frac{2}{3q}S_{jt}^{\ast }({\widehat{J_{M}})}%
\text{ }_{i}^{t}\right) \omega ^{ji}+{\left\Vert \nabla J_{M}\right\Vert }%
^{2}=0
\end{equation*}%
\begin{equation*}
-\left( S_{jt}({J_{M})}_{i}^{t}+\frac{2}{3q}S_{jt}^{\ast }({\widehat{J_{M}})}%
\text{ }_{i}^{t}\right) \omega ^{ij}+{\left\Vert \nabla J_{M}\right\Vert }%
^{2}=0
\end{equation*}%
\begin{equation*}
-\left( S_{jt}({J_{M})}_{i}^{t}+\frac{2}{3q}S_{jt}^{\ast }({\widehat{J_{M}})}%
\text{ }_{i}^{t}\right) ({J_{M})}_{n}^{i}g^{nj}+{\left\Vert \nabla
J_{M}\right\Vert }^{2}=0
\end{equation*}%
\begin{equation*}
S_{jt}\left( ({J_{M})}_{i}^{t}({J_{M})}_{n}^{i}\right) +\frac{2}{3q}%
S_{jt}^{\ast }\left( ({\widehat{J_{M}})}\text{ }_{i}^{t}({J_{M})}%
_{n}^{i}\right) g^{nj}+{\left\Vert \nabla J_{M}\right\Vert }^{2}=0
\end{equation*}%
\begin{equation*}
S_{jt}\left( p({J_{M})}_{n}^{t}-\frac{3}{2}q{\delta }_{n}^{t}\right) +\frac{2%
}{3q}S_{jt}^{\ast }\left( \frac{3}{2}q{\delta }_{n}^{t}\right) g^{nj}+{%
\left\Vert \nabla J_{M}\right\Vert }^{2}=0
\end{equation*}%
\begin{equation*}
\left( pS_{jt}\left( {J_{M}}_{n}^{t}\right) -\frac{3}{2}qS_{jn}+S_{jn}^{\ast
}\right) g^{nj}+{\left\Vert \nabla J_{M}\right\Vert }^{2}=0
\end{equation*}%
\begin{equation*}
pS_{jt}({J_{M})}_{n}^{t}g^{nj}-\frac{3}{2}qS_{jn}g^{nj}+S_{jn}^{\ast }g^{nj}+%
{\left\Vert \nabla J_{M}\right\Vert }^{2}=0
\end{equation*}%
\begin{equation*}
pS_{jt}\omega ^{tj}-\frac{3}{2}qS_{c}+S_{c}^{\ast }+{\left\Vert \nabla
J_{M}\right\Vert }^{2}=0
\end{equation*}%
\begin{equation*}
-pS_{jt}\omega ^{jt}-\frac{3}{2}qS_{c}+S_{c}^{\ast }+{\left\Vert \nabla
J_{M}\right\Vert }^{2}=0
\end{equation*}%
\begin{equation*}
S_{c}^{\ast }=\frac{3}{2}qS_{c}+pS_{jt}\omega ^{jt}-{\left\Vert \nabla
J_{M}\right\Vert }^{2}.
\end{equation*}
\end{proof}

\section{\noindent Linear connections}

In this section, by employing the method proposed in \cite{Salimov1} for
anti-Hermitian manifolds we search for linear connections with torsion on an
almost metallic Hermitian manifold $(M_{2k},g,J_{M})$. We will be calling
these connections linear connections of the first type and of the second
type, respectively.

Following the method from \cite{Salimov1}, we have the following definition.

\begin{definition}
A linear connection $\widetilde{\nabla }_{X}Y=\nabla _{X}Y+S(X,Y)$ on an
almost metallic Hermitian manifold $(M_{2k},g,J_{M})$ satisfying $\widetilde{%
{\nabla }}\omega =0$ and $S_{J_{M}}(X,Y,Z)+S_{J_{M}}(X,Z,Y)=0$ is called a
linear connection of the first type\textit{, }where $S$ is a $(1,2)-$tensor
field, $\omega $ is the fundamental $2-$form and $%
S_{J_{M}}(X,Y,Z)=g(S(X,Y),J_{M}Z)$.
\end{definition}

For the covariant derivative of the fundamental $2-$form $\omega $ with
respect to $\widetilde{\nabla }$, we find

\begin{eqnarray}
(\widetilde{{\nabla }}_{X}\omega )(Y,Z) &=&\widetilde{{\nabla }}_{X}(\omega
(Y,Z))-\omega (\widetilde{{\nabla }}_{X}Y,Z)-\omega (Y,\widetilde{{\nabla }}%
_{X}Z)  \label{TCG17} \\
&=&{\nabla }_{X}(\omega (Y,Z))-\omega (\nabla _{X}Y+S(X,Y),Z)  \notag \\
&&-\omega (Y,\nabla _{X}Z+S(X,Z))  \notag \\
&=&{\nabla }_{X}(\omega (Y,Z))-\omega (\nabla _{X}Y,Z)-\omega (Y,\nabla
_{X}Z)  \notag \\
&&-\omega (S(X,Y),Z)-\omega (Y,S(X,Z))  \notag \\
&=&({\nabla }_{X}\omega )(Y,Z)-\omega (S(X,Y),Z)-\omega (Y,S(X,Z))  \notag \\
&=&({\nabla }_{X}\omega )(Y,Z)-g(J_{M}S(X,Y),Z)-g(J_{M}Y,S(X,Z))  \notag \\
&=&({\nabla }_{X}\omega )(Y,Z)+g(S(X,Y),J_{M}Z)-g(S(X,Z),J_{M}Y)  \notag \\
&=&({\nabla }_{X}\omega )(Y,Z)+S_{J_{M}}(X,Y,Z)-S_{J_{M}}(X,Z,Y)  \notag
\end{eqnarray}%
for any vector fields $X,Y,Z$ on $M_{2k}$. In view of the assumptions for $%
\widetilde{\nabla }$, from (\ref{TCG17}) we get%
\begin{eqnarray*}
S_{J_{M}}(X,Y,Z) &=&-\frac{1}{2}({\nabla }_{X}\omega )(Y,Z) \\
g(S(X,Y),J_{M}Z) &=&-\frac{1}{2}g(({\nabla }_{X}J_{M})Y,Z) \\
g(J_{M}S(X,Y),Z) &=&\frac{1}{2}g(({\nabla }_{X}J_{M})Y,Z) \\
J_{M}S(X,Y) &=&\frac{1}{2}({\nabla }_{X}J_{M})Y \\
S(X,Y) &=&\frac{1}{3q}\widehat{J_{M}}({\nabla }_{X}J_{M})Y,
\end{eqnarray*}%
i.e., the linear connection of the first type is given by $\widetilde{\nabla 
}=\nabla +\frac{1}{3q}\widehat{J_{M}}({\nabla }J_{M})$. We calculate%
\begin{eqnarray*}
(\widetilde{\nabla }_{X}g)(Y,Z) &=&X(g(Y,Z))-g(\widetilde{\nabla }%
_{X}Y,Z)-g(Y,\widetilde{\nabla }_{X}Z) \\
&=&X(g(Y,Z))-g(\nabla _{X}Y+\frac{1}{3q}\widehat{J_{M}}({\nabla }%
_{X}J_{M})Y,Z) \\
&&-g(Y,\nabla _{X}Z+\frac{1}{3}\widehat{J_{M}}({\nabla }_{X}J_{M})Z) \\
&=&(\nabla _{X}g)(Y,Z)-\frac{1}{3q}g(\widehat{J_{M}}({\nabla }_{X}J_{M})Y,Z)-%
\frac{1}{3q}g(Y,\widehat{J_{M}}({\nabla }_{X}J_{M})Z) \\
&=&-\frac{1}{3q}g(({\nabla }_{X}J_{M})J_{M}Y,Z)+\frac{1}{3q}g(\widehat{J_{M}}%
Y,({\nabla }_{X}J_{M})Z) \\
&=&\frac{1}{3q}g(J_{M}Y,({\nabla }_{X}J_{M})Z)+\frac{1}{3q}g(\widehat{J_{M}}%
Y,({\nabla }_{X}J_{M})Z) \\
&=&\frac{p}{3q}g(Y,({\nabla }_{X}J_{M})Z).
\end{eqnarray*}%
Hence, we get the following result.

\begin{theorem}
On an almost metallic Hermitian manifold $(M_{2k},g,J_{M})$, the \textit{%
linear }connection of the first type is given by 
\begin{equation*}
\widetilde{\nabla }=\nabla +\frac{1}{3q}\widehat{J_{M}}({\nabla }J_{M})
\end{equation*}%
and it is metric with respect to $g$ if and only if the almost metallic
Hermitian manifold $(M_{2k},g,J_{M})$ is a\textit{\ }metallic K\"{a}hler
manifold. In the case, the \textit{linear }connection of the first type and
the Levi-Civita connection coincides each other.
\end{theorem}

\begin{definition}
A linear connection $\widetilde{\nabla }_{X}Y=\nabla _{X}Y+S(X,Y)$ on an
almost metallic Hermitian manifold $(M_{2k},g,J_{M})$ satisfying $\widetilde{%
{\nabla }}\omega =0$ and $S_{J_{M}}(X,Y,Z)+S_{J_{M}}(Z,Y,X)=0$ is called a
linear connection of the second type.
\end{definition}

We can write%
\begin{eqnarray*}
({\nabla }_{X}\omega )(Y,Z)+S_{J_{M}}(X,Y,Z)-S_{J_{M}}(X,Z,Y) &=&0 \\
\left( \nabla _{Y}\omega \right) \left( Z,X\right) +S_{J_{M}}\left(
Y,Z,X\right) -S_{J_{M}}\left( Y,X,Z\right) &=&0 \\
\left( {\nabla }_{Z}\omega \right) \left( X,Y\right) +S_{J_{M}}\left(
Z,X,Y\right) -S_{J_{M}}\left( Z,Y,X\right) &=&0
\end{eqnarray*}%
from which, by virtue of $S_{J_{M}}(X,Y,Z)+S_{J_{M}}(Z,Y,X)=0$, it follows
that%
\begin{eqnarray}
{2}S_{J_{M}}\left( X,Y,Z\right) &=&\left( {\nabla }_{X}\omega \right) \left(
Y,Z\right) +\left( {\nabla }_{Y}\omega \right) \left( Z,X\right) +\left( {%
\nabla }_{Z}\omega \right) \left( X,Y\right)  \label{TCG16} \\
{2}g\left( S(X,Y),J_{M}Z\right) &=&d\omega (X,Y,Z)  \notag \\
-2g(J_{M}S(X,Y),Z) &=&d\omega (X,Y,Z).  \notag
\end{eqnarray}%
On an almost metallic K\"{a}hler manifold we get $S=0$, which means that $%
\widetilde{\nabla }=\nabla $. Hence, we have:

\begin{theorem}
If an almost metallic Hermitian manifold $(M_{2k},g,J_{M})$ is\textit{\
almost metallic }K\"{a}hler, the linear connection \textit{of the second
type is egual to }$\nabla $.\textit{\ }
\end{theorem}

If the almost metallic Hermitian manifold $(M_{2k},g,J_{M})$ is\textit{\ }%
nearly metallic K\"{a}hler, then (\ref{TCG16}) reduces to%
\begin{eqnarray*}
-2g(J_{M}S(X,Y),Z) &=&3\left( {\nabla }_{X}\omega \right) \left( Y,Z\right)
\\
g(J_{M}S(X,Y),Z) &=&-\frac{3}{2}g(({\nabla }_{X}J_{M})Y,Z) \\
J_{M}S(X,Y) &=&-\frac{3}{2}({\nabla }_{X}J_{M})Y \\
S(X,Y) &=&-\frac{1}{q}\widehat{J_{M}}({\nabla }_{X}J_{M})Y.
\end{eqnarray*}%
Thus, we get:

\begin{theorem}
\noindent If an almost metallic Hermitian manifold $(M_{2k},g,J_{M})$ is%
\textit{\ }nearly metallic K\"{a}hler, the linear connection \textit{of the
second type} is given by 
\begin{equation*}
\widetilde{\nabla }=\nabla -\frac{1}{q}\widehat{J_{M}}({\nabla }J_{M}).
\end{equation*}
\end{theorem}

\end{document}